\documentclass[11pt,a4paper,fullpage]{article}
\usepackage{amsmath,amsthm}
\usepackage{amssymb}

\newtheorem{theorem}{Theorem}
\newtheorem{proposition}[theorem]{Proposition}
\newtheorem{lemma}[theorem]{Lemma}

\title{{\bf An asymptotic on the logarithms of the relative class numbers of imaginary abelian number fields 
of prime conductor and large degree}}
\author{
St\'ephane R. LOUBOUTIN\\
Aix Marseille Universit\'e, CNRS, I2M,\\ 
Marseille, France\\
stephane.louboutin@univ-amu.fr}
\date{\today}

\begin{document}
\bibliographystyle{alpha}
\maketitle


\begin{abstract}
An asymptotic on the logarithms of the relative class numbers 
of the cyclotomic number fields of prime conductors $p$ is known. 
Here we give an asymptotic on the logarithms of the relative class numbers of
the imaginary abelian number fields of prime conductors $p$ and large degrees $m =(p-1)/d$ 
with $\phi(d)=o(\log p)$.
We also show that this asymptotic does not hold true anymore 
under the only slightly weaker restriction $\phi(d)=O(\log p)$.
\end{abstract}

\section{Introduction}
We refer the reader to \cite[Chapters 3, 4 and 11]{Was} for more background details. 
Let ${\mathbb K}$ be an imaginary abelian number field of degree $m$ and prime conductor $p\geq 3$, 
i.e. let ${\mathbb K}$ be an imaginary subfield of a cyclotomic number field ${\mathbb Q}(\zeta_p)$ 
(Kronecker-Weber's theorem). 
Hence $m$ is even and $d=(p-1)/m$ is odd.
Its Hasse unit index $Q_{\mathbb K}$ is equal to $1$ 
(see e.g. \cite[Example 5, page 352]{Lem}). 
Let $w_{\mathbb K}$ be its number of complex roots of unity. 
Then $w_{\mathbb K}=2p$ if ${\mathbb K}={\mathbb Q}(\zeta_p)$ 
but $w_{\mathbb K}=2$ if ${\mathbb K}\varsubsetneq {\mathbb Q}(\zeta_p)$ (see \cite[Exercise 2.3]{Was}).
Let ${\mathbb K}^+$ be the maximal real subfield of ${\mathbb K}$ of degree $m/2$ fixed by the complex conjugation. 
The class number $h_{{\mathbb K}^+}$ of ${\mathbb K}^+$ 
divides the class number $h_{\mathbb K}$ of ${\mathbb K}$ . 
The {\it relative class number} of ${\mathbb K}$ is defined by 
$h_{\mathbb K}^- =h_{\mathbb K}/h_{{\mathbb K}^+}$. 
For $\gcd (t,p)=1$, 
let $\sigma_t$ be the ${\mathbb Q}$-automorphism of ${\mathbb Q}(\zeta_p)$ 
defined by $\sigma_t(\zeta_p) =\zeta_p^t$. 
Then $t\mapsto\sigma_t$ a is canonical isomorphic 
from the multiplicative group $({\mathbb Z}/p{\mathbb Z})^*$ 
to the Galois group ${\rm Gal}({\mathbb Q}(\zeta_p)/{\mathbb Q})$. 
Set
$$H
:={\rm Gal}({\mathbb Q}(\zeta_p)/{\mathbb K}) 
=\{t\in ({\mathbb Z}/p{\mathbb Z})^*;\ \alpha\in {\mathbb K}\Rightarrow\sigma_t(\alpha)=\alpha\},$$ 
a subgroup of index $m$ and odd order $d=(p-1)/m$ of the cyclic group $({\mathbb Z}/p{\mathbb Z})^*$. 
Let $X_p$ denote the cyclic group of order $p-1$ of the Dirichlet characters modulo $p$. 
Let $X_p^-=\{\chi\in X_p;\ \chi(-1)=-1\}$ be the set of the $(p-1)/2$ odd Dirichlet characters modulo $p$.
Now, $-1\not\in H$ 
and we set 
$$X_p^-(H)
=\{\chi\in X_p^-;\ \chi_{/H}=1\}.$$ 
Then $\# X_p^-(H) =m/2$
and we have the relative class number formula
\begin{equation}\label{formulahrelK}
h_{\mathbb K}^-
=w_{\mathbb K}\left (\frac{p}{4\pi^2}\right )^{m/4}\prod_{\chi\in X_p^-(H)} L(1,\chi).
\end{equation}
Using \eqref{formulahrelK} and the arithmetic-geometric mean inequality, 
we obtain 
\begin{equation}\label{hrelK}
h_{\mathbb K}^-
\leq w_{\mathbb K}\left (\frac{pM(p,H)}{4\pi^2}\right )^{m/4},
\end{equation}
where $M(p,H)$ is the mean square value
$$M(p,H)
:=\frac{1}{\# X_p^-(H) }\sum_{\chi\in X_p^-(H)}\vert L(1,\chi)\vert^2.$$ 
These mean square values have been studied by using Dedekind sums $s(h,p)\in {\mathbb Q}$:

\begin{proposition}\label{MpHn}
(See 
\cite[Proposition 2 and Theorem 6]{LouBKMS56}).
Let $d\geq 1$ be an odd divisor of $p-1$, 
where $p\geq 3$ is an odd prime number. 
Let $H_d$ be the only subgroup of order $d$ 
of the multiplicative cyclic group $({\mathbb Z}/p{\mathbb Z})^*$.
Then
$$M(p,H_d)
=\frac{\pi^2}{6}
\left (
1
+\frac{N(H_d,p)}{p}
\right ),$$
where 
$$N(H_d,p)
:=12\sum_{h\in H_d}s(h,p) -p\in {\mathbb Q}.$$
Moreover, for $d>1$ the rational number $N(H_d,p) $ is an odd rational integer.
\end{proposition}

Using Proposition \ref{MpHn} with $d=1$ and recalling that $s(1,p) =\frac{(p-1)(p-2)}{12p}$, 
we recover \cite{W}:
\begin{equation}\label{Mp1}
M(p,\{1\})
:=\frac{2}{p-1}\sum_{\chi\in X_p^-}\vert L(1,\chi)\vert^2
=\frac{\pi^2}{6}\left (1-\frac{1}{p}\right )\left (1-\frac{2}{p}\right )
\leq\frac{\pi^2}{6}
\ \ \ \ \ \hbox{($p\geq 3$).}
\end{equation}
Using \eqref{hrelK} we obtain the upper bound
\begin{equation}\label{boundhpminusQzetap}
h_{{\mathbb Q}(\zeta_p)}^-
\leq 2p\left (\frac{pM(p,\{1\})}{4\pi^2}\right )^{(p-1)/4}
\leq 2p\left (\frac{p}{24}\right )^{(p-1)/4}.
\end{equation}
The only other value of $d$ for which an explicit closed formula for $M(p,H_d)$ is known is $d=3$: 

\begin{proposition}\label{thp3}
(See \cite[Theorem 1]{LouBPASM64}).
Let $p\equiv 1\pmod 6$ be a prime number. 
Let $H_3$ be the only subgroup of order $3$ of the multiplicative cyclic group $({\mathbb Z}/p{\mathbb Z})^*$. 
Let ${\mathbb K}$ be the imaginary subfield of degree $(p-1)/3$ of the cyclotomic number field ${\mathbb Q}(\zeta_p)$.
Then 
\begin{equation}\label{boundhpminusK3}
M(p,H_3)
=\frac{\pi^2}{6}\left (1-\frac{1}{p}\right )
\leq\frac{\pi^2}{6}
\text{ and }
h_{\mathbb K}^-
\leq 2\left (\frac{p}{24}\right )^{(p-1)/12}
\end{equation}
(compare with \eqref{Mp1} and \eqref{boundhpminusQzetap}, 
and note the misprint in the exponent in \cite[(8)]{LouBPASM64}).
\end{proposition}

By \cite[Theorem 1.1]{LMQJM}, we have the following generalization 
of \eqref{boundhpminusQzetap} and \eqref{boundhpminusK3}:
\begin{equation}\label{boundhpminusKd}
N(p,H_d)=o(1),
\ M(p,H_d)=\frac{\pi^2}{6}+o(1)
\text{ and } 
h_{\mathbb K}^-\leq 2\left (\frac{(1+o(1))p}{24}\right )^{m/4}
\end{equation}
as $p$ tends to infinity 
and ${\mathbb K}$ runs over the imaginary subfields of ${\mathbb Q}(\zeta_p)$ 
of large degrees $m=(p-1)/d$ with 
\begin{equation}\label{restriction1}
d\leq\frac{\log p}{3\log\log p}.
\end{equation} 
In \cite[Theorem 2.1]{MS} it is proved that \eqref{boundhpminusKd} still holds true under the weaker restriction 
\begin{equation}\label{restriction2}
\phi(d) =o(\log p).
\end{equation}
Restriction \eqref{restriction1} is almost optimal.
Indeed, if $p$ runs over the Mersenne primes $p=2^d-1$, $d\geq 3$ odd, 
and $H_d$ is the subgroup of order $d$ of $({\mathbb Z}/p{\mathbb Z})^*$ generated by $2$, 
then $d$ is prime and 
$$N(p,H_d)
=2p-(6d-3),
\ M(p,H_d) =\frac{\pi}{2}\left (1-\frac{2d-1}{p}\right)
\text{ and }
h_{\mathbb K}^-\leq 2\left (\frac{p}{8}\right )^{m/4}$$ 
by \cite[Theorem 5.4]{LMCJM}. 

Finally, for $d_0$ a square-free integer, 
let $\chi'$ denote the non-primitive Dirichlet modulo $d_0p$ induced by $\chi\in X_p^-$. 
By using estimation of the mean square value $M_{d_0}(p,H_d)$ of 
$\vert L(1,\chi')\vert^2$ for $\chi\in X_p^-(H)$, 
we improved on
\eqref{boundhpminusQzetap} in \cite[Th\'eor\`eme 6]{LouCMB36/37} 
and on \eqref{boundhpminusKd} in \cite[Theorem 1.1]{LMCJM}: 
for any given positive constant $C<4\pi^2$ and for $p\geq p_C$ large enough we have 
\begin{equation}\label{boundhpminusKdimproved}
k_{\mathbb K}^-
\leq w_{\mathbb K}\left (\frac{p}{C}\right )^{m/4}
\end{equation}
as ${\mathbb K}$ runs over the imaginary subfields of ${\mathbb Q}(\zeta_p)$ 
of large degrees $m=(p-1)/d$ with $d$ as in \eqref{restriction1}. 

In the present paper, we not simply improve on the upper bounds
\eqref{boundhpminusK3}, \eqref{boundhpminusKd} and \eqref{boundhpminusKdimproved}. 
With restrictions slightly lighter than \eqref{restriction1}, 
we prove an asymptotic behavior instead of an upper bound:

\noindent\frame{\vbox{
\begin{theorem}\label{mainth1}
As $p$ tends to infinity 
and ${\mathbb K}$ runs over the imaginary abelian number fields of conductor $p$ 
and large degrees $m=(p-1)/d$ with $d$ as in \eqref{restriction2}
we have an explicit asymptotic 
\begin{equation}\label{asymptotichKminus}
h_{\mathbb K}^- 
=w_{\mathbb K}\left (\frac{(1+o(1))p}{4\pi^2}\right )^{m/4}.
\end{equation}
Moreover, this asymptotic is fully explicit under the following stricter restriction 
\begin{equation}\label{restriction3}
d\leq\frac{\log p}{\log\log p}
\end{equation} 
\end{theorem}
}}

Conjecturally, for a given even degree $m$ 
there are infinitely many imaginary number fields ${\mathbb K}$ 
of conductor $p$ and relative class number 
$h_{\mathbb K}^-\gg (\sqrt p\log\log p)^{m/2}$, 
see \cite{BCE} and \cite{Jo} for proofs of this conjecture for $m=2$.
Hence, we do need some restriction on $d$ for \eqref{asymptotichKminus} to hold true. 
The following asymptotic \eqref{asymptotichKminusMersenne} shows that 
\eqref{asymptotichKminus} does not hold true 
under the restriction $\phi(d)=O(\log p)$ only slightly weaker 
than restriction \eqref{restriction2}: 

\noindent\frame{\vbox{
\begin{theorem}\label{mainth2}
Let $p=2^d-1$ range over the (conjecturally infinitely many) Mersenne primes, 
which implies that $d$ is prime. 
Let ${\mathbb K}$ be the imaginary field of conductor $p$ 
and large degree $m=(p-1)/d$ $=2\cdot\frac{2^{d-1}-1}{d}$. 
We have an explicit asymptotic 
\begin{equation}\label{asymptotichKminusMersenne}
h_{\mathbb K}^- 
=w_{\mathbb K}\left (\frac{(1+o(1))p}{\pi^2}\right )^{m/4}.
\end{equation}
\end{theorem}
}}

Whereas \eqref{asymptotichKminus} is an asymptotic on the logarithm of the relative class numbers, 
it seems hopeless to have an asymptotic on the relative class numbers themselves. 
Indeed, even in the much studied case of the cyclotomic fields of prime conductors, 
Kummer's conjecture asserts that 
$$h_{{\mathbb Q}(\zeta_p)}^-
=2p(1+o(1))\left (\frac{p}{4\pi^2}\right )^{(p-1)/4}.$$ 
However A. Granville showed in \cite{Granville} that this conjecture is false 
if the two well-known and widely believed Hardy-Littlewood's and Elliott-Halberstam's conjectures are true. 
See also \cite{Murty} for more recent results on Kummer's conjecture.

A lot of results have been obtained for relative class number of cyclotomic fields, 
e.g. see \cite{Deb}, \cite{KLMES} , \cite{Lepisto}, \cite{LZ}, \cite{M} and \cite{MM}.

However, we could not find results in the spirit of \eqref{asymptotichKminus} and \eqref{asymptotichKminusMersenne} 
in the literature.

\section{Proof of Theorem \ref{mainth1}}
To prove \eqref{asymptotichKminus} we adapt the method developed in \cite[Proof of Theorem 1]{MM}. 
Let us give some details of their method. 
Since $L(1,\chi)\neq 0$ for $\chi\in X_p^-$, 
by \eqref{formulahrelK}, 
there exists $r_p>0$ such that $L(s,\chi)\neq 0$ for $\vert s -1\vert < r_p$ and $\chi\in X_p^-$. 
As in \cite[Section 2]{MM} we take $r_p =2p^{-3}$. 
Hence $L(s,\chi)$ does not vanish in the union $\Omega_p$ 
of this disc and the half-plane $\sigma =\Re(s)>1$, 
by the infinite products of these $L$-functions. 
Therefore, in this simply connected open set $\Omega_p$ 
there exists an holomorphic determination of 
$$f_H(s)
:=\sum_{\chi\in X_p^-(H)}\log L(s,\chi)$$
such that 
$$\log L(s,\chi)
=\sum_{n\geq 2}\frac{\Lambda(n)}{\log n}\chi (n)n^{-s}
\text{ for $\Re(s)>1$},$$
where $\Lambda (n) =\log q$ if $n =q^k$ is a power of a prime $q$ and $\Lambda (n)=0$ otherwise 
(the von Mangoldt's function). 
The orthogonality relation 
\begin{equation}\label{orthogonality}
\varepsilon_H(n)
:=\frac{1}{\# X_p^-(H)}\sum_{\chi\in X_p^-(H)}\chi (n)
=\begin{cases}
1&\hbox{if $n\in H$,}\\
-1&\hbox{if $-n\in H$}\\
0&\hbox{otherwise}
\end{cases}
\end{equation}
gives 
$f_H(s)
=f_{H,+1}(s)-f_{H,-1}(s)$, 
where 
\begin{equation}\label{seriesfH}
f_{H,\varepsilon}(s)
=\frac{m}{2}\sum_{h\in H}\sum_{\substack{n\geq 2\\ n\equiv\varepsilon h\pmod p}}
\frac{\Lambda(n)}{\log n}n^{-s}
\ \ \ \ \ \text{(for $\Re(s)>1$ and $\varepsilon\in\{-1,+1\}$).}
\end{equation}
Now, by \eqref{formulahrelK} we have
\begin{equation}\label{hKminusfH(1)}
h_{\mathbb K}^- 
=w_{\mathbb K}\left (\frac{p}{4\pi^2}\exp\left (\frac{4}{m}f_H(1)\right )\right )^{m/4}.
\end{equation}

\subsection{Bound on $\vert f_H(1)\vert$ 
and proof of \eqref{asymptotichKminus} 
from \eqref{hKminusfH(1)}}\label{sectionboundfH(1)}
\noindent {\bf 1.} First, for $\sigma >1$ and $\xi\in (1,\sigma)$, we write 
\begin{equation}\label{fH(1)fH(sigma)}
f_H(1) 
=f_H(\sigma) -f_H'(\xi)(\sigma -1).
\end{equation} 

\noindent {\bf 2.} Second, bearing on \cite[Proof of Lemma 1]{MM} 
and using Montgomery and Vaughan's explicit bound on primes in arithmetic progressions, 
we will obtain in Lemma \ref{boundfH(sigma)} a bound of the type
\begin{equation}\label{fH(sigma)Cpd}
\vert f_H(\sigma)\vert\leq C_{p,d} -\log (\sigma -1)
\ \ \ \ \ \text{(for $1<\sigma\leq 1+1/(5\log 2)$)}.
\end{equation} 

\noindent {\bf 3.} Third, set $\sigma_0 =1+p^{-3}$. 
The closed disc $\vert s-\sigma_0\vert\leq 2p^{-3}$ is included in $\Omega_p$ 
and by \cite[Proof of Lemma 2]{MM} we have 
$$\Re (f_H(s)) 
=\log\prod_{\chi\in X_p^-(H)} \vert L(s,\chi)\vert
\leq\frac{1}{2}m\log p
\ \ \ \ \ \text{(for $\vert s-\sigma_0\vert\leq 2p^{-3}$)}.$$ 

\noindent {\bf 4.} Fourth, as in \cite[Lemma 3]{MM} 
we apply the Borel-Carath\'eodory lemma to $g(s):=f_H(s)-f_H(\sigma_0)$ 
which satisfies $g(\sigma_0)=0$ and $\Re (g(s))\leq M =m\log p$ for $\vert s-\sigma_0\vert\leq R=2p^{-3}$ 
to deduce that 
\begin{equation}\label{fHprime(xi)}
\vert f_H'(\xi)\vert
=\vert g'(\xi)\vert
\leq 8M/R
=4mp^3\log p
\ \ \ \ \ \text{(for $\vert \xi\ -\sigma_0\vert\leq R/2 =p^{-3}$)}.
\end{equation} 

\noindent {\bf 5.} Using \eqref{fH(1)fH(sigma)}, \eqref{fH(sigma)Cpd} and \eqref{fHprime(xi)}, 
we end with the upper bound
$$\vert f_H(1)\vert
\leq C_{p,d} -\log (\sigma -1) +8M(\sigma -1)/R
\ \ \ \ \ \text{( for $1<\sigma\leq\sigma_0 +R/2 =1+2p^{-3}$)},$$
which for the optimal choice $\sigma =1+\frac{R}{8M}$ yields 
$$\vert f_H(1)\vert
\leq C_{p,d}+\log\left ({8eM/R}\right )
=C_{p,d}+\log (4emp^3\log p)
\leq C_{p,d}+\log (4ep^4\log p).$$
We have thus proved:

\begin{lemma}
With $C_{p,d}$ as defined in Lemma \ref{boundfH(sigma)} below, 
we have 
\begin{equation}\label{fH(1)Cpd}
\vert f_H(1)\vert
\leq C_{p,d}+\log\left (4ep^4\log p\right ).
\end{equation} 
\end{lemma}

Now, first assume that $d$ is as in \eqref{restriction2}. 
Using \eqref{Cpd2} we get that 
$$\frac{4}{m}\vert f_H(1)\vert
=\frac{4d}{p-1}\vert f_H(1)\vert
=o(1)$$
as $p$ tends to infinity. 
The first assertion of Theorem \ref{mainth1} follows, by \eqref{hKminusfH(1)}.\\

Second, assume that $d$ is as in \eqref{restriction3}. 
Then 
$p^{1/d}\geq\log p$ 
and 
$C_{p,d}
\leq\frac{p}{2\log p}+\frac{3}{2}+\frac{3}{8}\log p$, 
by \eqref{Cpd1}.
Hence, we also get
$$\frac{4}{m}\vert f_H(1)\vert
=\frac{4d}{p-1}\vert f_H(1)\vert
=o(1)$$ 
as $p$ tends to infinity. 
The second assertion of Theorem \ref{mainth1} follows, by \eqref{hKminusfH(1)}.

\subsection{Proof of the bound on $\vert f_H(\sigma)\vert$ for $\sigma>1$ given in \eqref{fH(sigma)Cpd}}
In \cite{MM} the authors dealt only with cyclotomic fields, 
i.e. with the case $m=p-1$ and $H=\{1\}$. 
In the second step of the strategy above they prove the bound 
$\vert f_{\{1\}}(\sigma)\vert\leq 3+\log (p/(\sigma-1))$ for $1<\sigma\leq 5/2$. 
To obtain this bound they first prove that for $m=p-1$ and $H=\{1\}$, we have
$$\frac{m}{2}\sum_{h\in H}\sum_{\substack{q\geq 2\\ q\equiv \varepsilon h\pmod p}}
q^{-\sigma}
\leq e^{-1}+\log\frac{3}{2(\sigma-1)}
\ \ \ \ \ \text{(for $1<\sigma\leq 5/2$ and $\varepsilon\in\{-1,+1\}$)}$$ 
(sums over primes $q$). 
The key point is that in this sum the primality of $q$ implies $q\geq 2p-1$. 
This is no longer true for $d\geq 3$. 
For example, let $p$ be a prime of the form $p=p_0^2+p_0+1$, with $p_0$ a prime. 
Take $d=3$ and $H=H_3 =\{1,p_0,p_0^2\}$. 
For $q =p_0$ we get the term 
$\frac{p-1}{6}p_0^{-\sigma}\geq\frac{p-1}{6}p^{-\sigma/2}$ in the sum above
and this term is $\gg p^{1/2}$ in the range $1<\sigma\leq 1+p^{-5}$. 
This ruins the hope of having something as good as \cite[Lemma 1]{MM} for $d\geq 3$. 
Instead, we have:

\noindent\frame{\vbox{
\begin{lemma}\label{boundfH(sigma)}
For $p\geq 29$ a prime, 
$1<\sigma\leq 1+1/(5\log 2)$ 
and $H$ a subgroup of odd order $d$ of the multiplicative group $({\mathbb Z}/p{\mathbb Z})^*$
we have the following explicit bounds
$$0\leq f_{H,-1}(\sigma),\ f_{H,+1}(\sigma),\ \left\vert f_H(\sigma)\right\vert
\leq C_{p,d}-\log(\sigma-1),$$
where 
\begin{equation}\label{Cpd1}
C_{p,d}
=\frac{1}{2}p^{1-1/d}
+\frac{3}{2} 
+\frac{3}{8}\log p .
\end{equation}
Moreover, as $p$ tends to infinity and $\phi(d)=o(\log p)$ we can take the following non-explict expression
\begin{equation}\label{Cpd2}
C_{p,d}
=\frac{3}{2} 
+\frac{3}{8}\log p 
+o(m),
\text{ where }
m=(p-1)/d.
\end{equation}
\end{lemma}
}}

Since $f_H(s)=f_{H,+1}(s)-f_{H,-1}(s)$, 
it suffices to prove the bound for the $f_{H,\varepsilon}(\sigma)$'s 
with $\varepsilon\in\{\pm 1\}$.
To begin with we notice that \eqref{seriesfH} and $\Lambda(n)\leq\log n$ give
$$0
<f_{H,\varepsilon}(\sigma)
\leq\Sigma
:=\frac{m}{2}\sum_{h\in H}\sum_{\substack{n\geq 2\\ n\equiv\varepsilon h\pmod p}}
\frac{1}{n^\sigma}
\ \ \ \ \ \text{(for $\sigma>1$ and $\varepsilon\in\{-1,+1\}$).}$$
Now, to bound this double sum $\Sigma$ we decompose it into 3 sub-sums $\Sigma=\Sigma_1+\Sigma_2+\Sigma_3$ 
by splitting the domain of variation of its summation index $n$ into three parts, 
$\Sigma_1$ for $2\leq n<p$, $\Sigma_2$ for $p\leq n<2p$ and $\Sigma_3$ for $n\geq 2p$. 
We bound $\Sigma_1$ in \eqref{step3} and \eqref{step3bis}, $\Sigma_2$ in \eqref{step2} and $\Sigma_3$ in \eqref{step1}.

\subsubsection{The range $2\leq n<p$}\label{the range2<=n<p}\label{firstrange}
If $n\equiv \varepsilon h\pmod p$ for some $h\in H$ and $n\geq 2$, 
then $n^d\equiv \varepsilon^dh^d\equiv\varepsilon^d\equiv\pm 1\pmod p$
and $n\geq (p-1)^{1/d}$.
Noticing that $\# H =d$ and $md =p-1$, it follows that
\begin{equation}\label{step3}
\Sigma_1
:=\frac{m}{2}\sum_{h\in H}\sum_{\substack{2\leq n<p\\n\equiv\varepsilon h\pmod p}}
\frac{1}{n^\sigma}
\leq\frac{(p-1)^{1-1/d}}{2}
\leq\frac{1}{2}p^{1-1/d}
\ \ \ \ \ \text{(for $\sigma>1$ and $\varepsilon\in\{-1,+1\}$).}
\end{equation}

Using some of the results of the recent preprint \cite{MS}, 
we can improve on \eqref{step3}. 
This improvement shows that Theorem \ref{mainth1} still holds true with the restriction \eqref{restriction2} 
weaker than the restriction \eqref{restriction3}.

\begin{lemma}
Take $\varepsilon\in\{-1,+1\}$.
As $p$ tends to infinity and $\phi(d)=o(\log p)$ we have 
\begin{equation}\label{step3bis}
\Sigma_1
=o(m).
\end{equation}
\end{lemma}

\begin{proof}
We keep the notation of \cite{MS}.
Let $H'$ be the subgroup of order $2d$ of $({\mathbb Z}/p{\mathbb Z})^*$ generated by $-1$ and $H$. 
Then
 $$\Sigma_1
 \leq\frac{m}{2}\sum_{1\neq\lambda\in H'}\rho'(\lambda,p)^{-1}
 \leq\frac{m}{2}\sum_{1\neq\lambda\in H'}\rho(\lambda,p)^{-1},$$ 
 where for $\lambda\in ({\mathbb Z}/p{\mathbb Z})^*$ we set
 \begin{multline*}
 \rho(\lambda,p)
 :=\min\{rs\ :\ (r,s)\in {\mathbb N}\setminus\{(0,0)\},\ r\equiv\lambda s\pmod p\}\\
 \leq\rho'(\lambda,p)
 :=\min\{r\geq 1\ :\ r\equiv\lambda\pmod p\}.
 \end{multline*}
 Setting 
 $\vartheta (H',p)
 =\min_{1\neq\lambda\in H'}\rho(\lambda,p)$ 
 and noticing that $\phi(2d)=\phi(d)$, 
 we have 
 $\vartheta (H',p)\gg p^{1/\phi(d)}$, by \cite [Lemma 4.1]{MS}. 
 Consequently,
$$\sum_{1\neq\lambda\in H'}\rho(\lambda,p)^{-1}
 \ll\vartheta(H',p)^{-1/10}+dp^{-1/4}
 \ll p^{-1/(10\phi(d))}+dp^{-1/4}
 =\exp\left (-\frac{\log p}{10\phi(d)}\right )+dp^{-1/4},$$ 
 by \cite[Lemma 4.3]{MS} applied with $\alpha =1$, $\beta=1/5$ and $\kappa =1/10$.
The desired result follows.
\end{proof}

\subsubsection{The range $p\leq n<2p$}
Noticing that $\# H =d$ and $md =p-1$, we clearly have
\begin{equation}\label{step2}
\Sigma_2
:=\frac{m}{2}\sum_{h\in H}\sum_{\substack{p\leq n<2p\\n\equiv\varepsilon h\pmod p}}
\frac{1}{n^\sigma}
\leq\frac{p-1}{2p}
\leq\frac{1}{2}
\ \ \ \ \ \text{(for $\sigma>1$ and $\varepsilon\in\{-1,+1\}$).}
\end{equation}

\subsubsection{The range $n\geq 2p$}
To make our proof somewhat clearer than \cite[Proof of Lemma 1]{MM}, 
we state and prove two technical Lemmas:

\begin{lemma}\label{boundSc}
For $\alpha>1$, set $X_\alpha=2(\alpha+1)/(\alpha-1)$.
For $\beta>0$, $x\geq \exp(2/\beta)$, and $X\geq 2/\beta$, set 
$$S_\beta(x)
=\frac{1}{\sqrt x}\sum_{2\leq k\leq\beta\log x}x^{1/k}
\text { and }G_\alpha(X)
=\left (\int_{1/\beta}^{X/2}e^{t}\frac{{\rm d}t}{t^2}\right )
-\frac{\alpha-1}{X}e^{X/2}.$$
If $(\alpha-1)/(\alpha+1)\leq\beta$ and $G_\alpha(X_\alpha)\leq 0$, 
then $S_\beta(x)\leq\alpha$ for $x\geq\exp(2/\beta)$.\\ 
This holds true for 
 $\alpha=11/4$ and $\beta=1/\log(2)$. 
\end{lemma}

\begin{proof}
If $\exp(2/\beta)\leq x<\exp(3/\beta)$, then $S_\beta(x)=1\leq\alpha$. 
Now assume that $x\geq\exp(3/\beta)$.
Take $X=\log x\geq 2/\beta$. 
Since 
$$S_\beta(x)-1
=\frac{1}{\sqrt x}\sum_{3\leq k\leq\beta\log x} x^{1/k}
\leq\frac{1}{\sqrt x}\int_2^{\beta\log x}x^{1/t}{\rm d}t
=e^{-X/2}\int_2^{\beta X}e^{X/t}{\rm d}t
=Xe^{-X/2}\int_{1/\beta}^{X/2}e^{t}\frac{{\rm d}t}{t^2},$$
by the change of variable $t\mapsto X/t$, 
we have 
$$(S_\beta(x)-\alpha)\frac{e^{X/2}}{X}
=(S_\beta(x)-1)\frac{e^{X/2}}{X}
-\frac{\alpha-1}{X}e^{X/2}
\leq G_\alpha(X).$$
Since $G_\alpha'(X) =(X_\alpha-X)\frac{e^{X/2}}{2(\alpha-1)X^2}$ 
and 
$X_\alpha\geq 2/\beta$,
we have $G_\alpha(X)\leq G_\alpha(X_\alpha)$ for $X\geq 2/\beta$.
\end{proof}

\begin{lemma}\label{seriesintegral}
Let $N\geq 1$ be a given integer.
Let $a_n\in {\mathbb C}$, $n\geq N$. 
For $x\geq N$, set $A(x)=\sum_{N\leq n\leq x}a_n$.
Let $b\in{\mathcal C}^1([N,+\infty),{\mathbb C})$ be such that $\lim_{+\infty}A(x)b(x)=0$ 
and $\int_N^\infty A(t)b'(t){\rm d}t$ is convergent. 
Then $\sum_{n\geq N}a_nb(n)$ is convergent 
and its sum is equal to 
$-\int_N^\infty A(t)b'(t){\rm d}t$.
\end{lemma}

\begin{proof}
Setting $M =\lfloor x\rfloor$ and noticing that $A(t)=A(n)$ for $n\leq t<n+1$ and $A_N=a_N$, 
we have 
$$\int_N^xA(t)b'(t){\rm d}t
=\left (\sum_{n=N}^{M-1}\int_n^{n+1}A(n)b'(t){\rm d}t\right )+\int_M^xA(M)b'(t){\rm d}t
=-\left (\sum_{n=N}^Ma_nb(n)\right )+A(x)b(x).$$
\end{proof}

Now, set 
$$\pi(x;a,p) 
=\sum_{\substack{2\leq q\leq x\\ q\equiv a\pmod p}}1
\text{ and }
\Pi(x;a,p) 
=\sum_{\substack{2\leq n\leq x\\ n\equiv a\pmod p}}\frac{\Lambda (n)}{\log n}
=\sum_{\substack{2\leq q^k\leq x\\ q^k\equiv a\pmod p}}\frac{1}{k}$$
(sums over primes $q\geq 2$, over integers $n$ and over integers $k\geq 1$ and primes $q\geq 2$). 
As in \cite[Proof of Lemma 1]{MM} we will use Montgomery and Vaughan's bound 
\begin{equation}\label{boundpixpa}
\pi(x;p,a)
<\frac{2x}{(p-1)\log (x/p)}
\ \ \ \ \ \text{(for $x>p$).}
\end{equation}
The multiplicative group $({\mathbb Z}/p{\mathbb Z})^*$ being cyclic, 
the set $\{n\in {\mathbb Z};\ n^k\equiv a\pmod p\}$ is either empty 
or the disjoint union of $\gcd(k,p-1)$ arithmetic progressions of common difference $p$.
Hence, for $X\geq 1$ we have 
$\#\{n;\ 1\leq n\leq X,\ n^k\equiv a\pmod p\}
\leq (X/p+1)\gcd(k,p-1)
\leq k(X/p+1)$. 
Therefore, for $x\geq 4$ we have
$$0\leq\Pi(x;a,p) -\pi(x;a,p) 
=\sum_{2\leq k\leq\frac{\log x}{\log 2}}\frac{1}{k}\sum_{\substack{q\leq x^{1/k}\\ q^k\equiv a\pmod p}}1
\leq\sum_{2\leq k\leq\frac{\log x}{\log 2}}\left (\frac{x^{1/k}}{p}+1\right )$$
and for $2\leq x<4$ we have $\Pi(x;a,p) =\pi(x;a,p)$. 
Theefore, by Lemma \ref{boundSc} we have
\begin{equation}\label{Pipi}
0\leq\Pi(x;a,p) -\pi(x;a,p) 
\leq\frac{11\sqrt x}{4p}+\frac{3\log x}{2}
\ \ \ \ \ (x\geq 2).
\end{equation}
For a given $a\in{\mathbb Z}$, 
taking $N=2p$ and $b(t)=t^{-\sigma}$ and recalling that $md=p-1$, 
we obtain
$$\frac{md}{2}\sum_{\substack{q\geq 2p\\ q\equiv a\pmod p}}\frac{1}{q^\sigma}
\leq\frac{md\sigma}{2}\int_{2p}^\infty\frac{\pi (x;p,a)}{x^{\sigma+1}}{\rm d}x
\leq\sigma\int_{2p}^\infty\frac{1}{x^{\sigma}\log(x/p)}{\rm d}x
=\sigma p^{1-\sigma}\int_{(\sigma-1)\log 2}^\infty e^{-x}\frac{{\rm d}x}{x},$$
by Lemma \ref{seriesintegral}, \eqref{boundpixpa} 
and by the change of variables $x\mapsto\exp(x/(\sigma-1))$. 
For $\varepsilon\in\{-1,+1\}$ and $1<\sigma\leq 1+1/(5\log 2)$, 
noticing that $\# H=d$ and $\sigma p^{1-\sigma}\leq \sigma e^{1-\sigma}\leq 1$ for $p\geq 3$ and $\sigma>1$,
we get
$$\frac{m}{2}\sum_{h\in H}\sum_{\substack{q\geq 2p\\ q\equiv\varepsilon h\pmod p}}\frac{1}{q^\sigma}
\leq\int_{(\sigma-1)\log 2}^\infty e^{-x}\frac{{\rm d}x}{x}
=-\log(\sigma)+I((\sigma-1)\log 2)\leq -\log(\sigma-1),$$
since $a\mapsto I(a)$ increases with $a\in (0,1)$ and $I(1/5)\leq 0$,
where 
$$I(a)=-\log\log 2+\int_a^1(e^{-x}-1)\frac{{\rm d}x}{x}
+\int_1^\infty e^{-x}\frac{{\rm d}x}{x}.$$
Now, by Lemma \ref{seriesintegral} and \eqref{Pipi}, 
for $\sigma>1$ we have 
\begin{align*}
\frac{md}{2}\sum_{\substack{n\geq 2p\\n\equiv a\pmod p}}\frac{1}{n^\sigma}
-\frac{md}{2}\sum_{\substack{q\geq 2p\\ q\equiv a\pmod p}}\frac{1}{q^\sigma}
&\leq\frac{md\sigma}{2}\int_{2p}^\infty\frac{\Pi (x;p,a)-\pi (x;p,a)}{x^{\sigma+1}}{\rm d}x\\
&\leq\frac{p}{2}\int_{2p}^\infty\frac{\frac{11\sqrt x}{4p}+\frac{3\log x}{2}}{x^{2}}{\rm d}x
=\frac{11}{4\sqrt{2p}}+\frac{3\log(2p)}{8}+\frac{3}{8}
\end{align*}
where we used and $md=p-1$ and $\sigma/x^{\sigma-1}\leq 1$ for $\sigma>1$ and $x\geq\exp(1)$. 
Noticing that $\# H=d$ 
and that $\frac{11}{4\sqrt{2p}}+\frac{\log 8}{8}+\frac{3}{8}\leq 1$ for $p\geq 29$, 
for $\varepsilon\in\{-1,+1\}$, $p\geq 29$ and $1<\sigma\leq 1+1/(5\log 2)$ 
we end up with the bond
\begin{equation}\label{step1}
\Sigma_3
:=\frac{m}{2}\sum_{h\in H}\sum_{\substack{n\geq 2p\\n\equiv\varepsilon h\pmod p}}
\frac{1}{n^\sigma}
\leq 1+\frac{3}{8}\log p -\log(\sigma-1).
\end{equation}

\newpage
\section{Proof of Theorem \ref{mainth2}}
Here ${\rm Gal}({\mathbb Q}(\zeta_p)/{\mathbb K})$ is the subgroup $H=\{2^k;\ 0\leq k\leq d-1\}$ 
of order $d$ of the multiplicative group $({\mathbb Z}/p{\mathbb Z})^*$. 
For $\chi\in X_p^-(H)$ let $\chi'$ denote the character modulo $2p$ induced by $\chi$, 
i.e. $\chi'(n) =0$ is $n$ is even and $\chi'(n)=\chi(n)$ if $n$ is odd. 
In $\Omega_p$ we can still define an holomorphic determination of 
$$g_H(s)=\sum_{\chi\in X_p^-(H)}\log L(s,\chi')$$ 
such that 
$$\log L(s,\chi') 
=\sum_{\substack{n\geq 3\\n\text{ odd}}}\frac{\Lambda(n)}{\log n}\chi(n)n^{-s}
\text{ for }\Re(s)>1.$$
Hence, here we have 
$g_H(s) =g_{H,+1}(s)-g_{H,-1}(s),$ 
where 
$$g_{H,\varepsilon}(s)
=\frac{m}{2}\sum_{h\in H}\sum_{\substack{n\geq 3\text{ odd}\\ n\equiv\varepsilon h\pmod p}}
\frac{\Lambda(n)}{\log n}n^{-s}
\ \ \ \ \ \text{(for $\Re(s)>1$ and $\varepsilon\in\{-1,+1\}$).}$$
Since $\chi(2)=+1$, we have 
$L(s,\chi) =(1-2^{-s})^{-1}L(s,\chi')$. 
Therefore,
$L(1,\chi) =2L(1,\chi')$ 
and 
$$h_{\mathbb K}^-
=2\left (\frac{p}{\pi^2}\right )^{m/4}\prod_{\chi\in X_p^-(H)}L(1,\chi')
=2\left (\frac{p}{\pi^2}\exp\left (\frac{4}{m}g_H(1)\right )\right )^{m/4},$$
by \eqref{formulahrelK}.
Now we proceed with $g_H(s)$ as we did with $f_H(s)$ with the following modifications.
Since the characters $\chi'$ are modulo $2p$ the bound in Point 3 of subsection \ref{sectionboundfH(1)} 
must be changed to $\Re(g_H(s))\leq\frac{1}{2}m\log (2p)$ 
and we end up with the bound 
$\vert g_H(1)\vert\leq C_p+\log (4ep^4\log (2p))$, 
where $C_p$ is obtained by modifying the proof of Lemma \ref{boundfH(sigma)} in the following way:\\
In the range $3\leq n<p$ and $n$ odd 
we notice that $n\equiv h\pmod p$ for $h\in H=\{2^k;\ 0\leq k\leq d-1\}$ has no solution 
whereas $n\equiv -h\pmod p$ if and only if $n=p-2^k$ for some $k\in\{1,\cdots,d-1\}$.
Hence \eqref{step3} must be replaced by the much better bound
$$0
\leq\frac{m}{2}\sum_{h\in H}\sum_{\substack{2\leq n<p,\text{ $n$ odd}\\n\equiv\varepsilon h\pmod p}}
\frac{1}{n^\sigma}
\leq\frac{m}{2}\sum_{k=1}^{d-1}\frac{1}{p-2^k}
\leq\frac{md}{2(p-2^{d-1})}
=1
\ \ \ \ \ \text{(for $\sigma>1$ and $\varepsilon\in\{-1,+1\}$).}$$
In the range $p\leq n<2p$ and $n$ odd we 
have 
$$\sum_{h\in H}\sum_{\substack{p\leq n<2p,\text{ $n$ odd}\\n\equiv\varepsilon h\pmod p}}
\frac{1}{n^\sigma}
\leq\sum_{h\in H}\sum_{\substack{p\leq n<2p\\n\equiv\varepsilon h\pmod p}}
\frac{1}{n^\sigma}
\ \ \ \ \ \text{(for $\sigma>1$ and $\varepsilon\in\{-1,+1\}$).}$$
and we can use the bound \eqref{step2}.
In the same way, in the range $n\geq 2p$ and $n$ odd we can use the bound \eqref{step1}.
Hence, we end up with the bound 
$$0\leq g_{H,-1}(\sigma),\ g_{H,+1}(\sigma),\ \left\vert g_H(\sigma)\right\vert
\leq C_{p}-\log(\sigma-1),
\text{ where }
C_{p}
=\frac{5}{2} 
+\frac{3}{8}\log p .$$

\newpage
\section{Remarks on Theorem \ref{mainth1} }
An anonymous referee of a previous version of this paper 
allowed us to include here her/his proof of our Theorem \ref{mainth1} under the restriction \eqref{restriction3}
base on \cite{GS1} and \cite{Montgomery}. 
Her/his following proof is somewhat shorter, 
but uses more sophisticated tools
and gives less explicit results. 
Noticing that $\# X_p^-(H) =m/2$, 
by \eqref{formulahrelK} to obtain \eqref{asymptotichKminus} it suffices to prove that 
\begin{equation}\label{meanlogL1X}
\frac{1}{\# X_p^-(H)}\sum_{\chi\in X_p^-(H)}\log L(1,\chi) 
=o(1).
\end{equation} 

Now, by \cite{Montgomery} and \cite{GS1},
there exists a set ${\mathcal E}(p)$ of non-principal characters modulo $p$ 
of cardinal $\# {\mathcal E}(p)\ll p^{1/4}$ such that for all characters $\chi$ modulo $p$ 
such that for $\chi\not\in {\mathcal E}(p)$ we have 
\begin{equation}\label{logL1X}
\log L(1,\chi)
=\sum_{n\leq (\log p)^{100}}\frac{\Lambda (n)\chi(n)}{n\log n}
+O\left (\frac{1}{\log p}\right ).
\end{equation}

Indeed, on applying \cite[Theorem 1]{Montgomery} with $Q=p$, $T=p$ and $\sigma =1-1/40$, 
there are fewer than 
$(Q^2T)^{3(1-\sigma)}(\log Q)^{13} 
=p^{9/40}(\log p)^{13}\ll p^{1/4}$ 
primitive characters $\chi$ with conductor below $p$ 
for which the Dirichlet $L$-function $L(s,\chi)$ has a zero in the rectangle 
$1-1/40 =\sigma\leq\Re(s)\leq 1$ and $\vert\Im(s)\vert\leq p$. 
We then apply \cite[Lemma 8.2]{GS1}, 
to the characters $\chi\in X_p^-(H)$ 
not in this exceptional set, with the choice $s =\sigma=1$, $t=0$, $q=p$, $y=(\log p)^{100}$ 
and $\sigma_0=\alpha =1-1/40$
(notice that the rectangle $\sigma_0<\Re(z)\leq 1$ and $\vert\Im(z)-t\vert\leq y+3$ 
is included in the previous rectangle $1-1/40 =\alpha\leq\Re(s)\leq 1$ and $\vert\Im(s)\vert\leq p$ 
 $p$ large enough to have $y+3\leq p$). 
We have 
$$\sigma_1
=\min\left (\sigma_0+\frac{1}{\log y},\frac{\sigma+\sigma_0}{2}\right )
=\min\left (1-\frac{1}{40}+\frac{1}{100\log_2p},1-\frac{1}{80}\right )
=1-\frac{1}{40}+\frac{1}{100\log_2p}$$ 
(for $p$ large enough). 
Hence 
$$\frac{\log q}{(\sigma_1-\sigma_0)^2}y^{\sigma_1-\sigma}
=10^4\frac{(\log p)(\log_2p)^2}{(\log p)^{100(1-\sigma_1)}}
=10^4\frac{(\log p)(\log_2p)^2}{(\log p)^{5/2-\log_2p}}
=10^4\frac{(\log p)(\log_2p)^2}{(\log p)^{5/2}}\exp(1)
\ll\frac{1}{\log p}$$
and \eqref{logL1X} holds true 
for $\chi\in X_p^-(H)$ not in the previous exceptional set of cardinality $\ll p^{1/4}$. 

Now, for $\chi\in X_p^-(H)$ in this exceptional set we have explicit bounds 
$\frac{1}{\sqrt p}\ll L(1,\chi)\ll\log p$ 
if $\chi$ is quadratic 
(by the class number formula \eqref{formulahrelK} 
applied to the imaginary quadratic field ${\mathbb Q}(\sqrt{-p})$ 
and \cite{LouCRAS323}) 
and $\frac{1}{\log p}\ll\vert L(1,\chi)\vert \ll\log p$ otherwise 
(e.g. see \cite{LouActa62} and \cite{LouCRAS323}). 
Hence, 
noticing that for $x\geq 2$ we have
$$\left\vert\sum_{n\leq x}\frac{\Lambda (n)\chi(n)}{n\log n}\right\vert
\leq \sum_{n\leq x}\frac{1}{n}
\ll\log x,$$ 
we deduce that
\begin{equation}\label{logL1Xexceptional}
\log L(1,\chi)
=\sum_{n\leq (\log p)^{100}}\frac{\Lambda (n)\chi(n)}{n\log n}
+\begin{cases}
O(\log p )&\hbox{if $\chi\in X_p^-(H)$ is exceptional and quadratic}\\
O(\log_2 p )&\hbox{if $\chi\in X_p^-(H)$ is exceptional and not quadratic}
\end{cases}
\end{equation}

Using \eqref{logL1X}, \eqref{logL1Xexceptional} and the orthogonality relation \eqref{orthogonality}, 
and recalling that $\# X_p^-(H)=m/2=(p-1)/(2d)$
it follows that 
$$\frac{1}{\# X_p^-(H)}\sum_{\chi\in X_p^-(H)}\log L(1,\chi) 
=\sum_{n\leq (\log p)^{100}}
\frac{\Lambda (n)}{n\log n}\varepsilon_H(n)
+O\left (\frac{dp^{1/4}\log_2p}{p}\right )+O\left (\frac{1}{\log p}\right ).$$
Now we saw in subsection \ref{firstrange}
that the absolute value of the sum in this latter right hand side 
is less than or equal to $\# H/(p-1)^{1/d} =d/(p-1)^{1/d}$. 
Therefore \eqref{meanlogL1X} and hence \eqref{asymptotichKminus} 
hold true under the restriction $d\leq\frac{\log p}{\log\log p}$ assumed in \eqref{restriction3}.

\newpage
{\small
\bibliography{central}

}
\end{document}